\documentclass[12pt]{article}
\usepackage{a4wide}
\usepackage{amsmath,amsfonts,amssymb,amsthm}
\usepackage{epsfig}

\usepackage{color}
\usepackage{enumerate}

\usepackage{hyperref}
\usepackage{breakurl}

\definecolor{modra3}{rgb}{.1,.0,.4}
\hypersetup{colorlinks=true, linkcolor=modra3, urlcolor=modra3, citecolor=modra3, pdfpagemode=UseNone, pdfstartview=} %colored links, no rectangles, no bookmarks, zoom

\theoremstyle{plain}

\newtheorem{theorem}{Theorem}
\newtheorem{lemma}[theorem]{Lemma}
\newtheorem{observation}[theorem]{Observation}
\newtheorem{corollary}[theorem]{Corollary}
\newtheorem{claim}[theorem]{Claim}
\newtheorem{conjecture}[theorem]{Conjecture}
\newtheorem{problem}[theorem]{Problem}

\begin{document}

\title{Counterexample to an extension of the Hanani--Tutte theorem on the surface of genus $4$}

\author{
Radoslav Fulek\thanks{IST, Klosterneuburg, Austria;
\texttt{radoslav.fulek@gmail.com}. Supported by the People Programme (Marie Curie Actions) of the European Union's Seventh Framework Programme (FP7/2007-2013) under REA grant agreement no [291734] and by Austrian Science Fund (FWF): M2281-N35.}
  \and
Jan Kyn\v{c}l\thanks{Department of Applied Mathematics and Institute for Theoretical Computer Science, 
Charles University, Faculty of Mathematics and Physics, 
Malostransk\'e n\'am.~25, 118 00~ Praha 1, Czech Republic;
\texttt{kyncl@kam.mff.cuni.cz}. Supported by project 
16-01602Y of the Czech Science Foundation (GA\v{C}R) and by Charles University project 
UNCE/SCI/004.}
} %end of author

\date{}

\maketitle

%==============================================================================================

\begin{abstract}
We find a graph of genus $5$ and its drawing on the orientable surface of genus $4$ with every pair of independent edges crossing an even number of times. This shows that the strong Hanani--Tutte theorem cannot be extended to the orientable surface of genus $4$. As a base step in the construction we use a counterexample to an extension of the unified Hanani--Tutte theorem on the torus. 
\end{abstract}

%==============================================================================================
\section{Introduction}
\label{section_intro}

The Hanani--Tutte theorem~\cite{Ha34_uber,Tutte70_toward} is a classical result that provides an algebraic characterization of planarity with interesting theoretical and algorithmic consequences, such as a simple polynomial algorithm for planarity testing~\cite{Sch13_hananitutte}. The theorem has several variants, the strong and the weak variant are the two most well-known. The notion ``the Hanani--Tutte theorem'' refers to the strong variant.

\begin{theorem}[The (strong) Hanani--Tutte theorem{~\cite{Ha34_uber,Tutte70_toward}}]
\label{theorem_strong}
A graph is planar if it can be drawn in the plane so that no pair of independent edges crosses an odd number of times.
\end{theorem}

\begin{theorem}[The weak Hanani--Tutte theorem{~\cite{CN00_thrackles,PT00_which,PSS06_removing}}]
\label{theorem_weak}
If a graph $G$ has a drawing $\mathcal{D}$ in the plane where every pair of edges crosses an even number of times, then $G$ has a crossing-free drawing in the plane that preserves the cyclic order of edges at each vertex of~$\mathcal{D}$.
\end{theorem}

The weak variant earned its name because of its stronger assumptions; however, it does not directly follow from the strong variant since its conclusion is stronger than just planarity of $G$.
For sub-cubic graphs, the weak variant implies the strong variant, since in this case pairs of adjacent edges crossing oddly can be dealt with by a local redrawing in a small neighborhood of each vertex. See the survey by Schaefer~\cite{Sch13_hananitutte} for a deeper historical overview and other variants of the Hanani--Tutte theorem.

Recently a common generalization of both the strong and the weak variant has been discovered. 

\begin{theorem}[Unified Hanani--Tutte theorem{~\cite{FKP16_unified,PSS06_removing}}]
\label{theorem_stronger}
Let $G$ be a graph and let $W$ be a subset of vertices of $G$. Let $\mathcal{D}$ be a drawing of $G$ where every pair of edges that are independent or have a common endpoint in $W$ cross an even number of times. Then $G$ has a plane drawing where cyclic orders of edges at vertices from $W$ are the same as in $\mathcal{D}$.
\end{theorem}

The strong Hanani--Tutte theorem is obtained by setting $W=\emptyset$, the weak variant is obtained by setting $W=V(G)$.

Theorem~\ref{theorem_stronger} directly follows from the proof of the Hanani--Tutte theorem by Pelsmajer, Schaefer and {\v{S}}tefankovi{\v{c}}~\cite{PSS06_removing}. See~\cite{FKP16_unified} for a slightly simpler proof, which is based on case distinction of the connectivity of $G$ and uses the weak Hanani--Tutte theorem as a base case. 

Cairns and Nikolayevsky~\cite{CN00_thrackles} extended the weak Hanani--Tutte theorem to an arbitrary orientable surface.
Pelsmajer, Schaefer and {\v{S}}tefankovi{\v{c}}~\cite{PSS09_surfaces} extended it further to arbitrary nonorientable surface. The \emph{embedding scheme} of a drawing $\mathcal{D}$ on a surface $S$ consists of a cyclic order of edges at each vertex and a signature $+1$ or $-1$ assigned to every edge, representing the parity of the number of crosscaps the edge is passing through.

\begin{theorem}[The weak Hanani--Tutte theorem on surfaces{~\cite[Lemma 3]{CN00_thrackles}, \cite[Theorem 3.2]{PSS09_surfaces}}]
\label{theorem_weaksurface}
If a graph $G$ has a drawing $\mathcal{D}$ on a surface $S$ such that every pair of edges crosses an even number of times, then $G$ has an embedding on $S$ that preserves the embedding scheme of $\mathcal{D}$. 
\end{theorem}

Pelsmajer, Schaefer and Stasi~\cite{PSS09_pp} extended the strong Hanani--Tutte theorem to the projective plane, using the list of forbidden minors. Colin de Verdi{\`e}re et al.~\cite{CKPPT16_direct} recently provided an alternative proof, which does not rely on the list of forbidden minors.

\begin{theorem}[The (strong) Hanani--Tutte theorem on the projective plane{~\cite{CKPPT16_direct,PSS09_pp}}]
\label{theorem_pp_strong}
If a graph $G$ has a drawing on the projective plane such that every pair of independent edges crosses an even number of times, then $G$ has an embedding on the projective plane.
\end{theorem}

Whether the strong Hanani--Tutte theorem can be extended to some other surface than the plane or the projective plane has been an open problem. Schaefer and {\v{S}}tefankovi{\v{c}}~\cite{SS13_block} showed that a minimal counterexample to the strong Hanani--Tutte theorem on any surface must be $2$-connected.

%-------------------------------------------------------------------------------------------
\subsection{Our results}

Our main result is a counterexample to the extension of the strong Hanani--Tutte theorem on the orientable surface of genus $4$.

\begin{theorem}
\label{theorem_counter_genus4}
There is a graph of genus $5$ that has a drawing on the orientable surface of genus $4$ with every pair of independent edges crossing an even number of times.
\end{theorem}

Theorem~\ref{theorem_counter_genus4} disproves a conjecture of Schaefer and {\v{S}}tefankovi{\v{c}}~\cite[Conjecture 1]{SS13_block} that the $\mathbb{Z}_2$-genus of a graph is equal to its genus; but the question whether the Euler $\mathbb{Z}_2$-genus of a graph is equal to its Euler genus remains open.
  
As a base step in the construction, we use a counterexample to the extension of the unified Hanani--Tutte theorem on the torus. 

\begin{theorem}
\label{theorem_counter_unified_torus}
There is a graph $G$ with the following two properties.
\begin{enumerate}
\item[{\rm 1)}] The graph $G$ has a drawing $\mathcal{D}$ on the torus with every pair of independent edges crossing an even number of times, and with a set $W$ of four vertices such that every pair of edges with a common endpoint in $W$ crosses an even number of times.
\item[{\rm 2)}] There is no embedding of $G$ on the torus with the same cyclic orders of edges at the vertices of $W$ as in $\mathcal{D}$.
\end{enumerate}
\end{theorem}

In our proof of Theorem~\ref{theorem_counter_unified_torus} the graph $G$ is isomorphic to $K_{3,4}$. The graph in Theorem~\ref{theorem_counter_genus4} will be obtained by attaching three stars $K_{1,4}$ to a sufficiently large grid.

We prove Theorem~\ref{theorem_counter_unified_torus} and Theorem~\ref{theorem_counter_genus4} in Section~\ref{section_proofs}, after establishing some basic notation. In Section~\ref{section_consequences} we show how to extend the results to surfaces of higher genus. In Section~\ref{section_questions} we briefly discuss several related questions and open problems.

%============================================================================================

\section{Notation}

Refer to the monograph by Mohar and Thomassen~\cite{MT01_graphs} for a detailed introduction into surfaces and graph embeddings.
By a {\em surface} we mean a connected compact $2$-dimensional topological manifold. Every surface is either {\em orientable} (has two sides) or {\em nonorientable} (has only one side). Every orientable surface $S$ is obtained from the sphere by attaching $g \ge 0$ \emph{handles}, and this number $g$ is called the {\em genus} of $S$.
Similarly, every nonorientable surface $S$ is obtained from the sphere by attaching $g \ge 0$ \emph{crosscaps}, and this number $g$ is called the {\em (nonorientable) genus} of $S$. The simplest orientable surfaces are the sphere (with genus $0$) and the torus (with genus $1$). The simplest nonorientable surfaces are the projective plane (with genus $1$) and the Klein bottle (with genus $2$). We denote the orientable surface of genus $g$ by $M_g$.

We will also consider \emph{surfaces with holes}: an \emph{orientable surface of genus $g$ with $k$ holes}, denoted by $M_{g,k}$, is obtained from $M_g$ by removing $k$ disjoint open discs whose boundaries are also disjoint. The boundaries of the removed discs thus belong to $M_{g,k}$ and they are called the \emph{boundary components} of $M_{g,k}$.

Let $G=(V,E)$ be a graph with no multiple edges and no loops,
and let $S$ be a surface or a surface with holes. 
A \emph{drawing} of $G$ on $S$ is a representation of $G$ where every vertex is represented by a unique point in $S$ and every
edge $e$ joining vertices $u$ and $v$ is represented by a simple curve in $S$ joining the two points that represent $u$ and $v$. If it leads to no confusion, we do not distinguish between
a vertex or an edge and its representation in the drawing and we use the words ``vertex'' and ``edge'' in both contexts. We require that in a drawing no edge passes through a vertex,
no two edges touch, every edge has only finitely many intersection points with other edges and no three edges cross at the same inner point. In particular, every common point of two edges is either their common endpoint or a crossing.

A drawing of $G$ on $S$ is an \emph{embedding} if no two edges cross. A \emph{face} of an embedding of $G$ on $S$ is a connected component of the topological space obtained from $S$ by removing all the edges and vertices of $G$. A \emph {$2$-cell} embedding is an embedding whose each face is homeomorphic to an open disc. In particular, a graph that has a $2$-cell embedding must be connected, but not necessarily $2$-connected.

The \emph{rotation} of a vertex $v$ in a drawing of $G$ on an orientable surface is the clockwise cyclic order of the edges incident to $v$. We will represent the rotation of $v$ by the cyclic order of the other endpoints of the edges incident to $v$. The {\em rotation system} of a drawing is the set of rotations of all vertices. 

A \emph{facial walk} corresponding to a face $f$ in a $2$-cell embedding of $G$ on an orientable surface is the closed walk $w(f)$ in $G$ with the following properties: the image of $w$ in the embedding forms the boundary of $f$, and whenever $w$ is entering a vertex $v$ along an edge $e$, the next edge on $w(f)$ is the edge that immediately follows $e$ in the rotation of $v$. In particular, while tracing the walk $w(f)$ in the embedding, the face $f$ is always on the left-hand side. 

The {\em Euler characteristic\/} of a surface $S$ of genus $g$, denoted by $\chi(S)$, is defined as $\chi(S)=2-2g$ if $S$ is orientable, and $\chi(S)=2-g$ if $S$ is nonorientable. Equivalently, if $v,e$ and $f$ denote the number of vertices, edges and faces, respectively, of a $2$-cell embedding of a graph on $S$, then $\chi(S)=v-e+f$.

We say that two edges in a graph are \emph{independent} if they do not share a vertex.
An edge in a drawing is {\em even\/} if it crosses every other edge an even number of times.
A vertex $v$ in a drawing is {\em even\/} if all the edges incident to $v$ cross each other an even number of times.
A drawing of a graph is \emph{even} if all its edges are even.
A drawing of a graph is \emph{independently even} if every pair of independent edges in the drawing crosses an even number of times.

The \emph{genus} $g(G)$ of a graph $G$ is the minimum $g$ such that $G$ has an embedding on $M_g$. The \emph{$\mathbb{Z}_2$-genus} of a graph $G$ is the minimum $g$ such that $G$ has an independently even drawing on $M_g$.

%============================================================================================
\section{Counterexamples}
\label{section_proofs}

%---------------------------------------------------------------------------------------------
\subsection{Proof of Theorem~\ref{theorem_counter_unified_torus}}

Let $G=K_{3,4}$. Let $V(G)=U\cup W$ where $U$ and $W$ are the two maximal independent sets, $U=\{1,2,3\}$, and $W=\{A,B,C,D\}$. We claim that the drawing $\mathcal{D}$ in Figure~\ref{obr_1_torus_unified_ABCD} satisfies the theorem. 

\begin{figure}
\begin{center}
\epsfbox{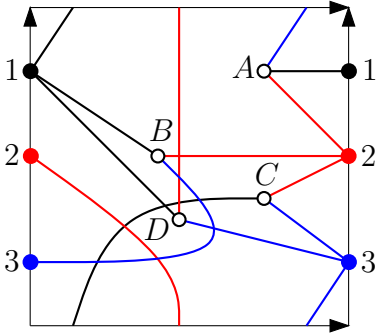}
\end{center}
\caption{An independently even drawing $\mathcal{D}$ of $K_{3,4}$ on the torus.  
The vertices of $W$ are drawn as empty circles, and each of them has rotation $(1,2,3)$.}
\label{obr_1_torus_unified_ABCD}
\end{figure}

Condition 1) is easily verified by inspection of the figure: the pairs $\{C1,D1\}, \{B2,D2\}$ and $\{B3,D3\}$ cross once, but they are all adjacent with a common vertex outside $W$, the pairs $\{C1,B3\}$ and $\{D2,B3\}$ cross twice, and no other pair of edges has a common crossing. To verify condition 2), we use the fact that every vertex of $W$ has the same rotation in $\mathcal{D}$; namely, $(1,2,3)$. Let $\mathcal{E}$ be an embedding of $G$ on an orientable surface $S$ such that the rotation of every vertex from $W$ is $(1,2,3)$. Assume without loss of generality that $S$ has minimum possible genus, which implies that $\mathcal{E}$ is a $2$-cell embedding. Since $G$ is bipartite, every face of $\mathcal{E}$ is bounded by a walk of even length. Moreover, we have the following crucial observation.

\begin{observation}\label{obs_aspon6cykly}
No face of $\mathcal{E}$ is bounded by a walk of length $4$.
\end{observation}

\begin{proof}
Suppose that $\mathcal{E}$ has a face bounded by a walk $v_1v_2v_3v_4$. Since $G$ is $2$-connected, the walk forms a $4$-cycle in $G$. By symmetry, we may assume that $v_1,v_3\in U$ and $v_2,v_4\in W$. It follows that in the rotation of $v_2$, the vertex $v_1$ is immediately followed by $v_3$, but in the rotation of $v_4$ the vertex $v_3$ is immediately followed by $v_1$; see Figure~\ref{obr_2_4face}. Thus, the rotations of $v_2$ and $v_4$ cannot be the same, but this is a contradiction with the definition of $\mathcal{E}$.
\end{proof}

\begin{figure}
\begin{center}
\epsfbox{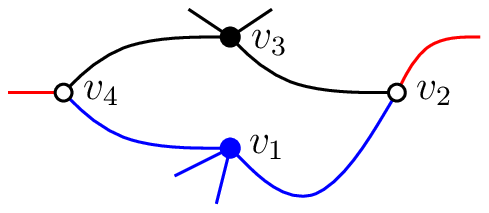}
\end{center}
\caption{A face bounded by a $4$-cycle would force different rotations of two vertices from $W$.}
\label{obr_2_4face}
\end{figure}

It follows that every face of $\mathcal{E}$ is bounded by a walk of length at least $6$. Let $v,e$ and $f$ be the numbers of vertices, edges and faces, respectively, of $\mathcal{E}$. We thus have $2e\ge 6f$, and so we can bound the Euler characteristic of $S$ as follows: 
\[
\chi(S)=v-e+f=\frac{1}{3}(3v-3e+3f)\le\frac{1}{3}(3v-2e)=\frac{1}{3}(21-24)=-1.
\]
This implies that the genus of $S$ is at least $\lceil(2+1)/2\rceil = 2$.

%---------------------------------------------------------------------------------------------
\subsection{Proof of Theorem~\ref{theorem_counter_genus4}}

\begin{figure}
\begin{center}
\epsfbox{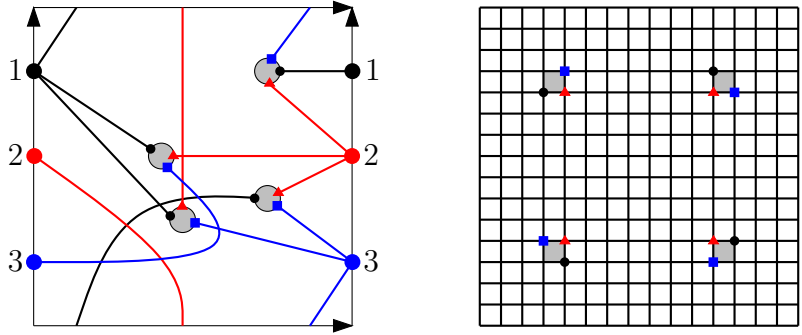}
\end{center}
\caption{Left: an independently even drawing $\mathcal{D}'$ of $3K_{1,4}$ on $M_{1,4}$. Right: a drawing $\mathcal{H}$ of the grid $H$ on $M_{0,4}$. Only three chosen vertices on each boundary component are marked.}
\label{obr_3_split_3K_1_4__grid}
\end{figure}

The proof is structured as follows. First we construct a graph $K$, by attaching $3K_{1,4}$ to a large grid. Then we verify that $K$ satisfies all the conditions of the theorem: it has an independently even drawing on $M_4$, it has no embedding on $M_4$, and it has an embedding on $M_5$.

%- - - - - - - - - - - - - - - - - - - - - - - - - - - - - 
\paragraph{A construction of the graph $K$.}
Let $G=K_{3,4}$, with parts $U$ and $W$, be the graph from the previous subsection and let $\mathcal{D}$ be the drawing of $G$ on the torus in Figure~\ref{obr_1_torus_unified_ABCD}. Cut a small circular hole around each vertex of $W$ in $\mathcal{D}$ and place a new vertex on all twelve intersections of an edge of $\mathcal{D}$ and a boundary of a hole; see Figure~\ref{obr_3_split_3K_1_4__grid}, left. In this way we obtain an independently even drawing $\mathcal{D}'$ of the disjoint union of three copies of $K_{1,4}$ on $M_{1,4}$, the torus with four holes. We consider the three copies of $K_{1,4}$ colored black, red and blue. On each boundary component of $M_{1,4}$, the clockwise order of the three vertices of $\mathcal{D}'$  is consistent with the rotations of the vertices of $W$ in $\mathcal{D}$: a black vertex is followed by a red vertex, which is followed by a blue vertex.

Let $N$ be a sufficiently large integer, and assume for convenience that $N$ is divisible by $8$. 
Let $H$ be the $N\times N$ grid; that is, a graph with vertex set $[N]\times[N]$ and the edge set $\{\{(i,j),(i',j')\}; ((i=i') \wedge (j=j'+1)) \vee ((i=i'+1) \wedge (j=j'))\}$. Let $\mathcal{H}$ be the canonical embedding of $H$ in the plane as a part of the integer grid, with edges drawn along the vertical and horizontal lines. Choose four \emph{special} $4$-cycles in $\mathcal{H}$ that are sufficiently separated from each other and also from the boundary of the grid; for example, the $4$-cycles with bottom left corners at 
%$(\lfloor N/4\rfloor ,\lfloor N/4\rfloor)$, $(\lfloor 3N/4\rfloor,\lfloor N/4\rfloor)$, $(\lfloor N/4\rfloor,\lfloor 3N/4\rfloor)$ and $(\lfloor 3N/4\rfloor,\lfloor 3N/4\rfloor)$. 
$(N/4, N/4)$, $( 3N/4, N/4)$, $( N/4, 3N/4)$ and $( 3N/4, 3N/4)$. 
See Figure~\ref{obr_3_split_3K_1_4__grid}, right. For each of these special $4$-cycles in $\mathcal{H}$, remove its interior from the plane, select three of its vertices, and mark them as black, blue, and red in clockwise order. We can now regard $\mathcal{H}$ as an embedding of $H$ on $M_{0,4}$.

Let $K$ be the graph obtained from $H$ by adding three vertices, labeled by $1$, $2$ and $3$, each of degree $4$, with vertex $1$ joined by an edge to each black vertex in $H$, vertex $2$ joined to each red vertex in $H$, and vertex $3$ joined to each blue vertex in $H$.

\begin{claim}
\label{claim_indep_even_M4}
The graph $K$ has an independently even drawing on $M_4$. 
\end{claim}

\begin{proof}
Such a drawing is obtained by gluing the drawings $\mathcal{D}'$ and $\mathcal{H}$ along the four boundary components of $M_{1,4}$ and $M_{0,4}$, respectively, in such a way that pairs of vertices of the same color are identified.
\end{proof}

%In the rest of the proof we show that $K$ has no embedding on $M_4$, and also observe that it does have an embedding on $M_5$.
%- - - - - - - - - - - - - - - - - - - - - - - - - - - - - 
\paragraph{Grid embedding lemma.}

We use the following grid embedding lemma by Geelen, Richter and Salazar~\cite{GRS04_grid}, which states that in every embedding of a large grid on a surface of fixed genus, a large portion of the grid is embedded in a planar way. This also follows from earlier statements by Thomassen~\cite[Proposition 3.2]{Tho97_excluded} or Mohar~\cite[Theorem 5.1]{Mo92_width}, and is implicit in the proof by Robertson and Seymour~\cite{RoSe90_VIII_Kuratowski} that each surface has only finitely many forbidden minors.

\begin{lemma}[{\cite[Lemma 4]{GRS04_grid}}]
\label{lemma_grid}
Suppose that $H$ is an $N\times N$ grid embedded on $M_g$, and let $t,k$ be positive integers such that $N\ge t(k+1)$ and $t^2\ge 2g+1$. Then a $k \times k$ subgrid $H'$ of $H$ is embedded in a topological disc in $M_g$ whose boundary is formed by the boundary $4(k-1)$-cycle of $H'$.
\end{lemma}

The idea of the proof of Lemma~\ref{lemma_grid} is roughly the following. It is enough to show that in every embedding of the grid $H$ on $M_g$, the number of noncontractible $4$-cycles of $H$ is bounded by a function of $g$; in this case linear in $g$. This is a consequence of an elementary topological result stating that a collection $\mathcal{C}$ of disjoint noncontractible closed curves in $M_g$ that appears on the boundary of a common component of $M_g\setminus \bigcup\mathcal{C}$, has cardinality at most $2g$. Indeed, such a collection cannot contain three pairwise homotopic curves, and a collection of pairwise disjoint pairwise nonhomotopic noncontractible closed curves in $M_g$ has cardinality at most $g$, which can be shown by induction on $g$.

%\begin{claim}
%The graph $K$ has no embedding on $M_4$. 
%\end{claim}

%Suppose that $K$ has an embedding $\mathcal{K}$ on $M_4$. 
%
%- - - - - - - - - - - - - - - - - - - - - - - - - - - - - 
\paragraph{The graph $K$ has no embedding on $M_4$.}
Let $\mathcal{K}$ be an embedding of $K$ on an orientable surface $S$ of minimum possible genus; in particular, $\mathcal{K}$ is a $2$-cell embedding.

Let $H_0$ be a subgraph of $H$ induced by the the vertices $(\mathbb{N}\cap[3N/8,5N/8])\times (\mathbb{N}\cap[3N/8,5N/8])$. The graph $H_0$ is a square grid, sufficiently far from the special $4$-cycles and from the boundary of $H$ if $N$ is large enough. We use Lemma~\ref{lemma_grid} for the induced embedding of $H_0$ with $k\ge 5$. Let $H'$ be the resulting square grid and let $D' \subset S$ be the smallest topological closed disc containing the image of $H'$ in $\mathcal{K}$. We chose $k\ge 5$ so that $H'$ has at least $16$ vertices on its perimeter, which is sufficient for our next construction, although in Figure~\ref{obr_4__K_4_5} we draw a slightly larger grid.

%For simplicity, we use Lemma~\ref{lemma_grid} with a sufficiently large $N$ and do not try to optimize its value. Let $H'$ be the grid obtained from the lemma 

\begin{figure}
\begin{center}
\epsfbox{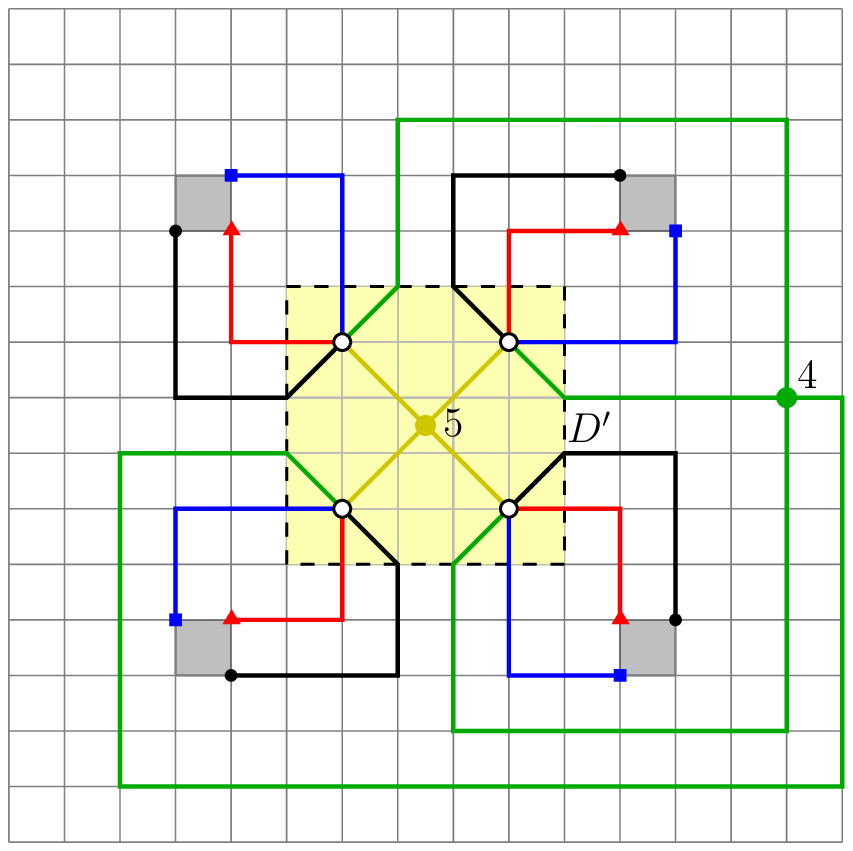}
\end{center}
\caption{A part of an orientable embedding of $K_{4,5}$ constructed from $\mathcal{K}$, using vertex-disjoint paths in $H$ and a disc $D'$, whose boundary is dashed in the figure. The $4$-cycles of $H$ within $D'$ are guaranteed to be filled with topological discs. Each of the four vertices of $W'$, represented by circles, has the rotation $(1,2,3,4,5)$.
}
\label{obr_4__K_4_5}
\end{figure}

Using the embedding $\mathcal{K}$ and the disc $D'$ we construct an embedding $\mathcal{E}'$ of $G'=K_{4,5}$ on $S$; see Figure~\ref{obr_4__K_4_5}. Let $V(G')=U'\cup W'$ where $U'$ and $W'$ are independent sets of size $5$ and $4$, respectively, and $U'=\{1,2,3,4,5\}$. We will refer to the vertices $1,2,3,4$ and $5$ together with their incident edges as black, red, blue, green and yellow, respectively. We identify the vertices $1,2,3$ of $U'$ with the vertices $1,2,3$, respectively, of $\mathcal{K}$. We place the green vertex $4$ in $H$ so that it is sufficiently far from $H'$, from the four special $4$-cycles, and from the boundary of $H$. We place the yellow vertex $5$ and all the four vertices of $W'$ inside $D'$; not necessarily coinciding with any vertices of~$\mathcal{K}$. 

We draw the black, red and blue edges of $G'$ along the edges of $\mathcal{K}$ incident with vertices $1,2$ and $3$, respectively, so that they reach the special $4$-cycles in $H$. Then we extend these twelve edges using vertex-disjoint paths in $H$, until they reach the boundary of $D'$. We draw the last portions of these edges inside $D'$, without having to use the embedding $\mathcal{K}$. Due to the planarity of $H$, the cyclic orders of the colored vertices on the boundaries of the special $4$-cycles are ``linked'' by the black, red and blue edges to opposite cyclic orders around the vertices of $W'$ inside $D'$; in particular, the rotation of each vertex of $W'$ in the constructed drawing contains the cyclic subsequence $(1,2,3)$. Moreover, we can make sure that the black and blue edges, incident to the vertices $1$ and $3$, respectively, are ``accessible'' from the boundary of $H$, while the red edges are ``hidden''. We proceed similarly with the green edges, which are drawn from the green vertex $4$ along vertex-disjoint paths of $H$ (and disjoint from the black, red and blue edges drawn previously) until they reach the boundary of $D'$, and continue inside $D'$. Such vertex-disjoint paths supporting the black, red, blue and green edges exist since the degree of the vertex $4$ in $H$ is $4$ and the disc $D'$, the special $4$-cycles, the boundary of $H$, and the vertex $4$ are all sufficiently far from each other in $H$. The yellow edges are drawn completely inside $D'$. The crucial property that we satisfy is that the rotations of the four vertices in $W'$ are all equal to $(1,2,3,4,5)$. This implies the following observation, analogous to Observation~\ref{obs_aspon6cykly} in the previous subsection.

\begin{observation}\label{obs_10}
No face of $\mathcal{E}'$ is bounded by a walk of length shorter than $10$. In fact, the length of each facial walk in $\mathcal{E}'$ is divisible by $10$.
\end{observation}

\begin{proof}
Since $\mathcal{E}'$ is an embedding on an orientable surface, whenever we trace a facial walk in the counterclockwise direction, an edge from $i\in U'$ to $w\in W'$ must be followed by the edge from $w$ to $i+1$ (taken modulo $5$). 
\end{proof}

Let $v,e$ and $f$ be the numbers of vertices, edges and faces, respectively, of $\mathcal{E}'$. By Observation~\ref{obs_10}, we have $2e\ge 10f$, and so we can bound the Euler characteristic of $S$ as follows: 
\[
\chi(S)=v-e+f=\frac{1}{5}(5v-5e+5f)\le\frac{1}{5}(5v-4e)=\frac{1}{5}(45-80)=-7.
\]
This implies that the genus of $S$ is at least $\lceil(2+7)/2\rceil = 5$. Therefore, $K$ has no embedding on $M_4$. 

%- - - - - - - - - - - - - - - - - - - - - - - - - - - - - 
\paragraph{The graph $K$ has an embedding on $M_5$.}
%Finally, we observe that $K$ has an embedding on $M_5$. 
First we describe an embedding of $G$ on $M_2$ where the rotations of the four vertices $A,B,C,D$ in $W$ are equal. We can embed $G$ so that the rotation of each vertex in $W$ is $(1,2,3)$, the rotation of $1$ is $(C,A,B,D)$, and the rotations of both $2$ and $3$ are $(A,B,C,D)$. This embedding has three faces, with facial walks of lengths $6,6$ and $12$. By Euler's formula, this rotation system indeed gives an embedding on $M_2$. An embedding of $K$ on $M_5$ is obtained by an analogous gluing operation as the drawing in Claim~\ref{claim_indep_even_M4}.

%===========================================================================================
\section{Consequences}
\label{section_consequences}
Using the additivity of the genus~\cite{BHKY62_additivity} and the $\mathbb{Z}_2$-genus~\cite{SS13_block} of a graph over its components, by taking the disjoint union of the graph $K$ from Theorem~\ref{theorem_counter_genus4} with $k$ copies of $K_5$ we obtain a counterexample to an extension of the strong Hanani--Tutte theorem on an arbitrary orientable surface of genus larger than $4$. Moreover, by taking $k$ disjoint copies of $K$, we obtain a separation of the genus and the $\mathbb{Z}_2$-genus by a multiplicative factor of $5/4$.

\begin{corollary}
\label{cor_separation_5_4}
For every positive integer $k$ there is a graph of genus $5k$ and $\mathbb{Z}_2$-genus at most $4k$.
\end{corollary}

Schaefer and {\v{S}}tefankovi{\v{c}}~\cite{SS13_block} asked whether the genus of a graph can be bounded by a function of its $\mathbb{Z}_2$-genus. In our follow-up paper~\cite{FK18_Z2genus}, we show that this follows from a folklore unpublished Ramsey-type result about unavoidable graph minors of large genus.

%===========================================================================================
\section{Related questions and open problems}
\label{section_questions}

Schaefer~\cite{Sch13_hananitutte} introduced the following weaker variant of the strong Hanani--Tutte theorem, parametrized by a positive integer $t$.

\begin{conjecture}[{\cite[Conjecture 3.3]{Sch13_hananitutte}}]
\label{conjecture_parametrized}
If a graph $G$ has an independently even drawing on a surface $S$ such that every pair of edges crosses at most $t$ times, then $G$ has an embedding on $S$.
\end{conjecture}

The drawing on $M_4$ constructed in Claim~\ref{claim_indep_even_M4} disproves Conjecture~\ref{conjecture_parametrized} for $t=2$ and orientable surfaces of genus at least $4$. The conjecture remains open for $t=1$. In this case, only adjacent edges of $G$ are allowed to cross in the initial drawing. In fact, the following question, often expressed as ``do adjacent crossings mater?'', is open.

\begin{problem}
\label{problem_adjacent}
Let $S$ be a surface other than the plane or the projective plane. Assume that a graph $G$ has a drawing on $S$ where only adjacent edges are allowed to cross. Does $G$ have an embedding on $S$?
\end{problem}

Problem~\ref{problem_adjacent} can also be formulated in terms of the \emph{independent crossing number} of $G$ on $S$, which may be denoted by $\mathrm{cr}_{-S}(G)$ using the notation in Schaefer's survey on crossing numbers~\cite{Sch17_survey}. For a given surface $S$, Problem~\ref{problem_adjacent} then asks whether $\mathrm{cr}_{-S}(G)=0$ implies $\mathrm{cr}_S(G)=0$.

The strong Hanani--Tutte theorem and its possible generalizations can be weakened in several other ways. For example, instead of an independently even drawing of $G$ we may consider the edges of $G$ oriented and require a drawing of $G$ where for every pair of independent edges $e$ and $f$, the number of crossings in which $e$ crosses $f$ from the left is equal to the number of crossings in which $e$ crosses $f$ from the right. This can be formulated in terms of the \emph{independent algebraic crossing number} of $G$ on $S$, denoted by $\mathrm{iacr}_S(G)$~\cite{Sch17_survey}, which is well-defined on orientable surfaces.

\begin{problem}
\label{problem_algebraic}
Let $g\ge 1$. Does $\mathrm{iacr}_{M_g}(G)=0$ imply $\mathrm{cr}_{M_g}(G)=0$?
\end{problem}

For the graph $K$ from the proof of Theorem~\ref{theorem_counter_genus4} we can only show that $\mathrm{iacr}_{M_4}(K)\le 2$.

In the drawing of $K$ from Claim~\ref{claim_indep_even_M4}, there are three pairs of adjacent edges such that for each of the pairs, the union of the two edges contains a noncontractible curve. Can this be avoided in a counterexample? This question was suggested to us by Jeff Erickson.

\begin{problem}
Let $S$ be a surface other than the plane or the projective plane. Assume that a graph $G$ has an independently even drawing on $S$ where the union of every pair of adjacent edges can be covered by a topological disc. Does $G$ have an embedding on $S$?
\end{problem}

%============================================================================================

%%%%%%%%%%%%%%%%%%%%%%%%%%%%%%%%%%%%%%%%%%%%%%%%%%%%%%%%%%%%%%%%%%%%%%%%%%%%%%%%%%%%%%%%%


\begin{thebibliography}{10}

\bibitem{BHKY62_additivity}
J. Battle, F. Harary, Y. Kodama and J. W. T. Youngs, 
Additivity of the genus of a graph,
{\em Bull. Amer. Math. Soc.} {\bf 68} (1962), 565--568. 

\bibitem{CN00_thrackles}
G.~Cairns and Y.~Nikolayevsky, Bounds for generalized thrackles, 
{\em Discrete Comput. Geom.} {\bf 23}(2) (2000), 191--206.

\bibitem{CKPPT16_direct}
{\'E}. Colin de Verdi{\`e}re, V. Kalu{\v{z}}a, P. Pat{\'a}k, Z. Pat{\'a}kov{\'a} and M. Tancer,
A direct proof of the strong Hanani--Tutte theorem on the projective plane,
{\em J. Graph Algorithms Appl.} {\bf 21}(5) (2017), 939--981.

\bibitem{FK18_Z2genus}
R. Fulek and J. Kyn\v{c}l,
The $\mathbb{Z}_2$-genus of Kuratowski minors,
{\em Proceedings of the 34th International Symposium on Computational Geometry (SoCG 2018), Leibniz International Proceedings in Informatics (LIPIcs)} 99, 40:1--40:14, Schloss Dagstuhl--Leibniz-Zentrum fuer Informatik, 2018.

\bibitem{FKP16_unified}
R. Fulek, J. Kyn\v{c}l and D. P{\'a}lv{\"o}lgyi,
Unified Hanani--Tutte theorem,
{\em Electron. J. Combin.} {\bf 24}(3) (2017), P3.18, 8 pp.

\bibitem{GRS04_grid}
J. F. Geelen, R. B. Richter and G. Salazar,
Embedding grids in surfaces,
{\em European J. Combin.} {\bf 25}(6) (2004), 785--792.

\bibitem{Ha34_uber}
H.~Hanani, {\"{U}}ber wesentlich unpl{\"{a}}ttbare {K}urven im drei-dimensionalen {R}aume, 
{\em Fundamenta Mathematicae\/} {\bf 23} (1934), 135--142.

\bibitem{Mo92_width}
B.~Mohar,
Combinatorial local planarity and the width of graph embeddings,
{\em Canad. J. Math.} {\bf 44}(6) (1992), 1272--1288.

\bibitem{MT01_graphs}
B.~Mohar and C.~Thomassen, 
{\em Graphs on surfaces\/}, 
Johns Hopkins Studies in the Mathematical Sciences, Johns Hopkins University Press, Baltimore, MD (2001), ISBN 0-8018-6689-8.

\bibitem{PT00_which}
J.~Pach and G.~T{\'o}th, Which crossing number is it anyway?, 
{\em J. Combin. Theory Ser. B\/} {\bf 80}(2) (2000), 225--246.

\bibitem{PSS09_pp}
M.~J. Pelsmajer, M.~Schaefer and D. Stasi, 
Strong Hanani--Tutte on the projective plane,
{\em SIAM J. Discrete Math.} {\bf 23}(3) (2009), 1317--1323. 

\bibitem{PSS06_removing}
M.~J. Pelsmajer, M.~Schaefer and D.~{\v{S}}tefankovi{\v{c}}, Removing even crossings, 
{\em J. Combin. Theory Ser. B\/} {\bf 97}(4) (2007), 489--500.

\bibitem{PSS09_surfaces}
M.~J. Pelsmajer, M.~Schaefer and D.~{\v{S}}tefankovi{\v{c}},
Removing even crossings on surfaces,
{\em European J. Combin.} {\bf 30}(7) (2009), 1704--1717. 

\bibitem{RoSe90_VIII_Kuratowski}
N. Robertson and P. D. Seymour, 
Graph minors. VIII. A Kuratowski theorem for general surfaces,
{\em J. Combin. Theory Ser. B\/} {\bf 48}(2) (1990), 255--288. 

\bibitem{Sch13_hananitutte}
M.~Schaefer, 
Hanani-{T}utte and related results, {\em Geometry---Intuitive, Discrete, and Convex}, vol.~24 of {\em Bolyai Soc. Math. Stud.}, 259--299, J\'anos
Bolyai Math. Soc., Budapest (2013).

\bibitem{Sch17_survey}
M. Schaefer, 
The graph crossing number and its variants: A survey, 
{\em Electron. J. Combin.}, Dynamic Survey 21 (2017).

\bibitem{SS13_block}
M. Schaefer and D.~{\v{S}}tefankovi{\v{c}},
Block additivity of $\mathbb{Z}_2$-embeddings,
{\em Graph Drawing, Lecture Notes in Computer Science\/} 8242, 185--195, Springer, Cham, 2013.

\bibitem{Tho97_excluded}
C. Thomassen,
A simpler proof of the excluded minor theorem for higher surfaces,
{\em J. Combin. Theory Ser. B} {\bf 70}(2) (1997), 306--311. 

\bibitem{Tutte70_toward}
W.~T. Tutte, Toward a theory of crossing numbers, 
{\em J. Combinatorial Theory\/} {\bf 8} (1970), 45--53.

\end{thebibliography}
\end{document}